\documentclass[reqno]{amsart}
\usepackage[T1]{fontenc}
\usepackage{dsfont}
\usepackage{mathrsfs}
\usepackage[colorlinks, citecolor=blue, linkcolor=blue]{hyperref}%³¬Á´½Ó
\usepackage{xcolor}
\usepackage[a4paper,asymmetric]{geometry}
\usepackage{mathscinet}
\usepackage{latexsym}
\usepackage{amsthm}
\usepackage{amssymb}
\usepackage{amsfonts}
\usepackage{amsmath}
\usepackage{longtable}
\usepackage{graphicx}
\usepackage{multirow}
\usepackage{multicol}
\usepackage{amsfonts, amsmath}
\usepackage{latexsym,bm,amsfonts,amssymb,pifont,mathbbol,bbm}
\usepackage{verbatim}
\usepackage{caption}
\usepackage{graphicx}
\usepackage{subfigure}

\newtheorem{theorem}{Theorem}[section]
\newtheorem{thm}[theorem]{Theorem}

\newtheorem{lem}[theorem]{Lemma}
\newtheorem{remark}[theorem]{Remark}
\newtheorem{proposition}[theorem]{Proposition}
\newtheorem{prop}[theorem]{Proposition}
\newtheorem{corollary}[theorem]{Corollary}
\newtheorem{hyp}[theorem]{HYPOTHESIS}
\theoremstyle{definition}

\newtheorem{defn}[theorem]{Definition}
\newtheorem{example}[theorem]{Example}
 
\theoremstyle{remark}
\numberwithin{equation}{section}

 \DeclareMathAlphabet{\mathpzc}{OT1}{pzc}{m}{it}
 \DeclareMathAlphabet{\mathsfsl}{OT1}{cmss}{m}{sl}

\newcommand{\dif}{\mathrm{d}}

\newcommand{\Be}{\begin{equation}}
\newcommand{\Ee}{\end{equation}}
\newcommand{\Bs}{\begin{split}}
\newcommand{\Es}{\end{split}}
\newcommand{\Bes}{\begin{equation*}}
\newcommand{\Ees}{\end{equation*}}
\newcommand{\BT}{\begin{thm}}
\newcommand{\ET}{\end{thm}}
\newcommand{\Bp}{\begin{proof}}
\newcommand{\Ep}{\end{proof}}
\newcommand{\BL}{\begin{lem}}
\newcommand{\EL}{\end{lem}}
\newcommand{\BP}{\begin{proposition}}
\newcommand{\EP}{\end{proposition}}
\newcommand{\BC}{\begin{corollary}}
\newcommand{\EC}{\end{corollary}}
\newcommand{\BR}{\begin{remark}}
\newcommand{\ER}{\end{remark}}
\newcommand{\BD}{\begin{defn}}
\newcommand{\ED}{\end{defn}}
\newcommand{\BI}{\begin{itemize}}
\newcommand{\EI}{\end{itemize}}

\allowdisplaybreaks

\begin{document}
\title{ Moment estimator for an AR(1) model with non-zero mean driven by a long memory Gaussian noise}
\author[]{Yanping LU}
 \address{School of Mathematics and Statistics, Jiangxi Normal University, Nanchang, 330022, Jiangxi, China}
\email{luyp@jxnu.edu.cn}
 %\author[]{}
 %\address{ School of Mathematical and Statistical Sciences, Arizona State University,\,Arizona, USA. $ }%(Corresponding author.)}
% \email {}
%-------------------------------------------------------------------------------------------------------
\begin{abstract}

In this paper, we consider an inference problem for the first order autoregressive process with non-zero mean driven by a long memory stationary Gaussian process. Suppose that the covariance function of the noise can be expressed as $|k|^{2H-2}$ times a positive constant when $k$ tends to infinity, and the fractional Gaussian noise and the fractional ARIMA model are special examples that satisfy this assumption. We propose moment estimators and prove the strong consistency, the asymptotic normality and joint asymptotic normality.\par


\end{abstract}
\maketitle

\section{Introduction}\label{I1}
For the first order autoregressive model $(X_t, t \in \mathbb{N})$ driven by a given noise sequence $\xi = (\xi_t, t \in \mathbb{Z})$:
\begin{equation}\label{AR1}
X_t=\alpha+\theta X_{t-1} +\xi_t, \quad t \in \mathbb{N}
\end{equation}with $X_0 = 0$, the inference problem regarding the parameter $\theta$ has been extensively studied in probability and statistics literatures. For when $\xi$ is independent identical distribution or a
martingale difference sequence, this problem has been widely studied over the past decades (see
\cite{Andeson1979}, \cite{Lai1983} and the references therein). Parameter estimation of AR(1) model driven by i.i.d. sequence
with non-zero mean was considered in \cite{Mami2013}.

The maximum likelihood estimator of AR(p) was investigated in \cite{Brouste2014} for regular stationary
Gaussian noise, in which they transform the observation model into an "equivalent" model with
Gaussian white noise. In \cite{Brouste2014}, it is pointed out that for the strongly dependent noises the least
square estimator is generally not consistent. In case of long-memory noise, the detection of a
change of the above parameter $\theta$ is studied by means of the likelihood ratio test \cite{Brouste2020}. \cite{Chen2023} consider
second moment estimator of $\theta$ in AR(1) model with zero mean driven by a long memory Gaussian
noise.

In this paper, we will discuss the long-range dependence Gaussian noise case and propose moment estimators of $\theta$ and $\alpha$. First, we find that it is very convenient to construct moment estimators when we restrict the domain of the parameter $\theta$ in
$$\Theta=\{ \theta \in \mathbb{R} ~|~ 0<\theta<1 \}.$$
It seems that this restriction is very reasonable for real-world context sometimes. In fact, $|\theta|< 1$ is an assumption to ensure the model \eqref{AR1} to have a stationary solution. We rule out the case of $-1 < \theta < 0$ in which the series tends to oscillate rapidly. We also rule out the case of $\theta=0$ in which $X_t$ is not an autoregressive model any more.

Next, we assume that the stationary Gaussian noise $\xi$ satisfies the following Hypothesis~\ref{hyp1}:
\begin{hyp}\label{hyp1}
The covariance function $R_{\xi}(k) = E(\xi_0 \xi_k)$ for any $k \in \mathbb{Z}$ satisfies
\begin{equation}\label{hyp11}
R_{\xi}(k) \rightarrow C|k|^{2H-2}, \quad \text{as} ~k \to \infty,
\end{equation}where $H\in (\frac 12,1)$ and $C$ is a position constant.
\end{hyp}
We mentioned that in order to simplify the calculation in this paper, we decided to omit the slowly varying function $L$; in fact, while this function is sometimes tedious to deal with from a mathematical point of view, it is of less importance from a philosophical point of view.

It is well-known that Eq.~\eqref{hyp11} is equivalent to the spectral density of $\xi$ satisfying
$$h_{\xi}(\lambda) \backsim C_H |\lambda|^{1-2H}, \quad \text{as}~\lambda \to 0,$$
with $C_H = \pi^{-1}\Gamma(2H-1) \sin(\pi -\pi H)$. Please refer to \cite{Beran2013} or Lemma 2.2 below.

We will see that the fractional Gaussian noise and the fractional ARIMA model driven by Gaussian white noise are special examples satisfying Hypothesis~\ref{hyp1}.

We will replace $X_t$ in an AR(1) model with zero mean to $X_t- \mu$, then
$$X_t-\mu =\theta(X_{t-1}-\mu)+\xi_t,\quad t \in \mathbb{N}.$$
Here, $\mu = \frac {\alpha} {1-\theta}$.

Denote
\begin{align}\label{Yt}
Y_t:=\sum_{j=0}^{\infty} \theta^j \xi_{t-j},
\end{align}When $|\theta| < 1$, the solution to the model \eqref{AR1} with initial value $X_0 = 0$ can be represented as:
\begin{align}\label{Xt}
X_t=\mu +Y_t +\theta^t \zeta,
\end{align}where $\mu + Y_t$ is the stationary solution and $\zeta$ is a normal random variable with zero mean.

It is clear that the second moment of $Y_t$ is:
\begin{equation*}
 f(\theta):= E(Y_t^2)=\sum_{i,j=0}^{\infty}\theta^{i+j} R_{\xi}(i-j).
\end{equation*}When $0 < \theta < 1$, $f(\theta)$ is positive and strictly increasing (see \cite{Chen2023} (5)). Denote $\bar{X} := \frac 1n \sum_{t=1}^{n} X_t$ and $f^{-1}(\cdot)$ is the inverse function of $f(\cdot)$.

We propose the second moment estimator of $\theta$ as:
\begin{align}\label{thetaestimator}
\hat{\theta}_n = f^{-1} \left(\frac 1n \sum_{t=1}^n(X_t-\bar{X})^2 \right)
\end{align}and the moment estimator of $\alpha$ as
\begin{align}\label{hatalpha}
\hat{\alpha}_n=(1-\hat{\theta}_n)\bar{X}.
\end{align}

In this paper, we will show the strong consistency and give the asymptotic distribution for $\hat{\alpha}_n$ and $\hat{\theta}_n$. Moreover, we also give the joint asymptotic normality. These results are stated in the following theorems:
\begin{thm}\label{consistent}
Under Hypothesis~\ref{hyp1}, the estimator $\hat{\theta}_n$ and $\hat{\alpha}_n$ is strongly consistent, i.e.,
\begin{equation*}
    \lim_{n\to \infty}\hat{\theta}_n=\theta \quad \text{a.s.}; \quad \quad \lim_{n\to \infty}\hat{\alpha}_n=\alpha \quad \text{a.s.}.
\end{equation*}
\end{thm}
\begin{thm}\label{asymptotic}
Under Hypothesis~\ref{hyp1} and suppose that $H\in (\frac 12, \frac 34 )$. Then, we have the following asymptotic distribution as $n\to \infty$:
\begin{enumerate}
  \item[(1)] asymptotic normality of $\hat{\theta}_n$
  \begin{equation*}
      \sqrt{n}(\hat{\theta}_n -\theta) \xrightarrow{law} \mathcal{N} \left(0, \frac{\sigma_H^2}{[f'(\theta)]^2} \right),
  \end{equation*}where $\sigma_H^2=2\sum_{k\in\mathbb{Z}}R_Y^2(k)$, $R_Y(k)=E(Y_t Y_{t+k})$ and $f'(\theta)$ is the derivative of $f(\theta)$;
  \item [(2)] asymptotic normality of $\hat{\alpha}_n$
  \begin{equation*}
      n^{1-H}(\hat{\alpha}_n-\alpha) \xrightarrow{law} \mathcal{N}(0,\sigma_1^2),
  \end{equation*}where $\sigma_1^2=[H(2H-1)]^{-1}\pi^{-1}B(2H-1,2-2H)\sin(2\pi H-\pi)$ and $B$ is the Beta function;
  \item [(3)] Denote $G_{n,1}:=\sqrt{n}(\hat{\theta}_n-\theta)$ and $G_{n,2}:=n^{1-H}(\hat{\alpha}_n-\alpha)$,  then we have the joint asymptotic normality $$(G_{n,1},G_{n,2})\xrightarrow{law}\mathcal{N}_2(0,\Sigma),$$ where
  $$\Sigma = \begin{pmatrix}
 \frac{\sigma_H^2}{[f'(\theta)]^2} & 0 \\
 0 & \sigma_1^2
\end{pmatrix}.$$
\end{enumerate}
\end{thm}Next, we give some sequences that satisfy Hypothesis~\ref{hyp1}.
\begin{example}
The fractional Gaussian noise with covariance function
$$R_{\xi}=\frac 12 (|k+1|^{2H}+|k-1|^{2H} - 2|k|^{2H}), \quad k\in \mathbb{Z}$$satisfies Hypothesis~\ref{hyp1} when $H > \frac 12$. It is well-known that $R_{\xi} \backsim H(2H -1)|k|^{2H-2}$ as $k \to \infty$.
\end{example}
\begin{example}
The fractional ARIMA(0,d,0) model driven by a Gaussian white noise satisfies Hypothesis~\ref{hyp1} when $d \in (0,\frac 12 )$. It is well-known that its the covariance function spectral density satisfy
$$R_{\xi}(k)=\sigma \frac{\Gamma(k+d)\Gamma(1-2d)}{\Gamma(k+1-d)\Gamma(1-d)\Gamma(d)},\quad k \ge 0 , $$
$$ h_{\xi}(\lambda)=\frac{\sigma}{2\pi}|2\sin(\lambda/2)|^{-2d}, \quad \lambda \to 0 , $$ where $\Gamma(\cdot)$ is Gamma function.
\end{example}
\begin{figure}[htbp]
    \centering
    \subfigure[a]{
    \includegraphics[width=5.5cm]{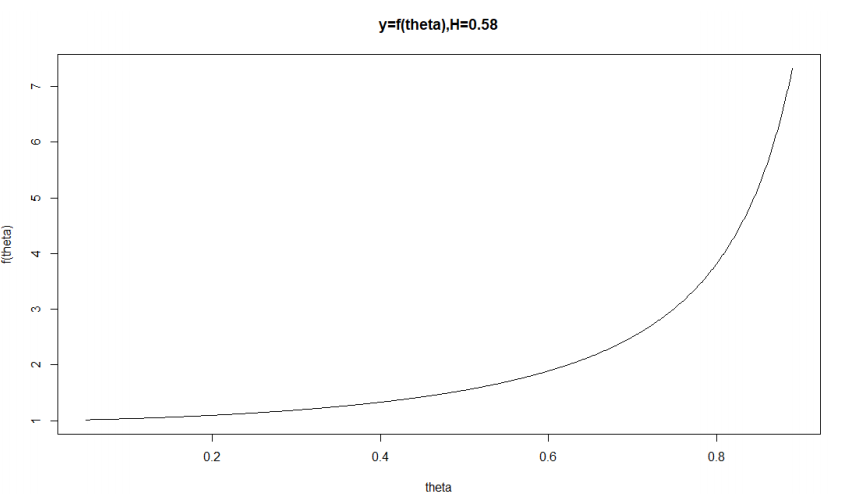}
    }
    \subfigure[b]{
    \includegraphics[width=5.5cm]{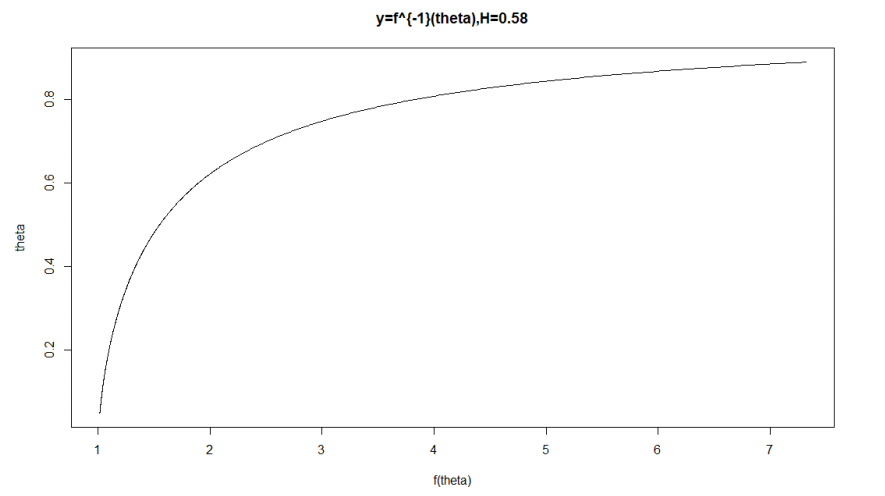}
    }

    \subfigure[c]{
    \includegraphics[width=5.5cm]{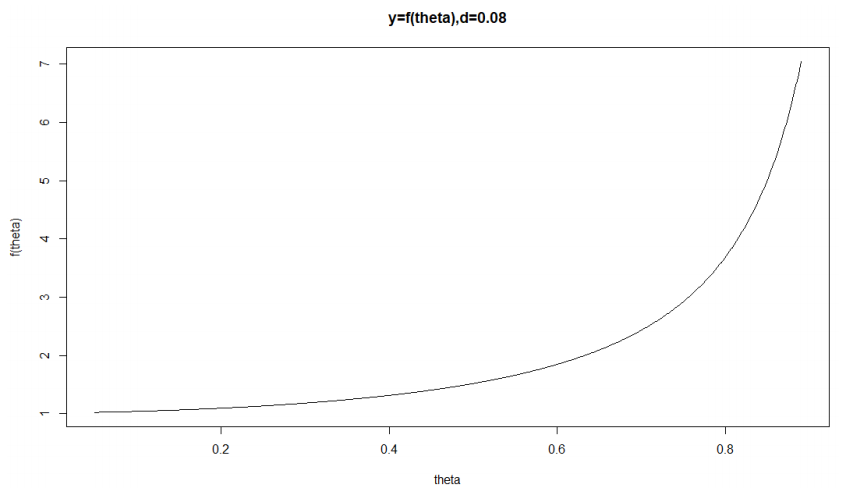}
    }
    \subfigure[d]{
    \includegraphics[width=5.5cm]{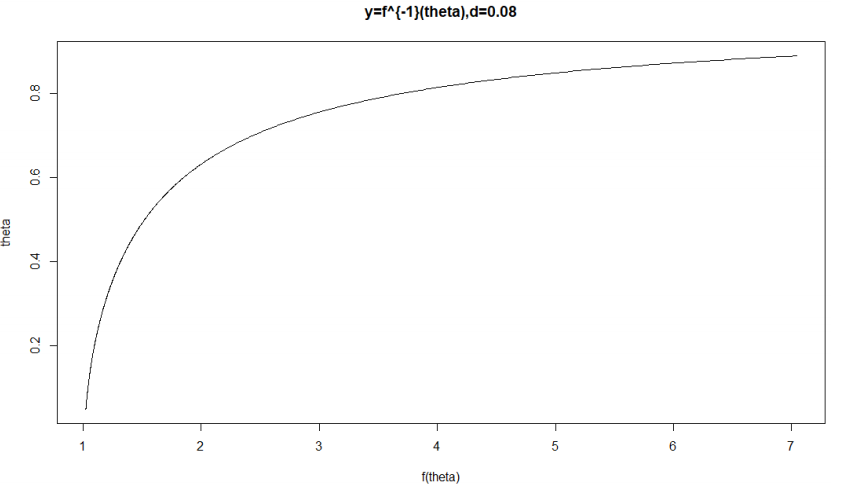}
    }
    \caption{The figures of the function $y = f(\theta)$ and the inverse function $\theta = f^{-1}(\theta)$ for fGn with Hurst parameter $H = 0.58$ ((a) and (b), respectively) and ARFIMA(0, d, 0) with $d = 0.08$ ((c) and (d), respectively) respectively.}
    \label{tu1}
\end{figure}
Here, we give the figure.\ref{tu1} of the function $y = f(\theta)$ in the interval $\theta \in (0, 1)$ and the inverse function $\theta = f^{-1}(\theta)$ with $y > 0$ for fGn with Hurst parameter $H = 0.58$ and ARFIMA(0, d, 0) with $d = 0.08$ respectively.

In the remainder of this paper, $C$ will be a generic positive constant independent of $n$ the value of which may differ from line to line.

\section{Preliminary}\label{I3}
In this section, we list the main definitions and theorems that is used to show our results. The following definition is cited from Definition 1.2 of \cite{Beran2013} respectively.
\begin{defn}\label{defn1}
Let $\{ \xi_t \}$ be a second-order stationary process with autocovariance function $R_{\xi}(k)(k \in \mathbb{Z})$ and spectral density
$$h_{\xi}(\lambda)=(2\pi)^{-1} \sum_{k=-\infty}^{\infty} R_{\xi}(k)\exp(-i k\lambda), \quad \lambda \in [-\pi , \pi].$$ Then $\{ \xi_t \}$ is said to exhibit linear long-range dependence if
$$h_{\xi}(\lambda)= C_H |\lambda|^{1-2H},$$ where $H \in (\frac 12 , 1)$.
\end{defn}
The following theorem is well-known and is cited from Theorem 1.3 of \cite{Beran2013}:
\begin{thm}
Let $R_{\xi}(k)(k \in \mathbb{Z})$ and $h(\lambda)(\lambda \in [-\pi, \pi])$ be the autocovariance function and spectral density respectively of a second-order stationary process $\{ \xi_t \}$. Then the following holds:
\begin{enumerate}
  \item[(1)] If
  \begin{equation*}
      R_{\xi}(k)=C|k|^{2H-2},\quad k\in \mathbb{Z},
  \end{equation*}where $H \in (\frac 12,1)$, then
  $$h(\lambda)\backsim C_H |\lambda|^{1-2H}, \quad \lambda \to 0,$$where
  $$C_H=\pi^{-1}\Gamma(2H-1)\sin(\pi - \pi H).$$
 \item[(2)] If
  $$h(\lambda)= C' |\lambda|^{1-2H}, \quad 0<\lambda <\pi,$$where $H \in (\frac 12,1)$, then
  \begin{equation*}
      R_{\xi}(k)\backsim C'_H|k|^{2H-2},\quad k\to \infty,
  \end{equation*}
  where
  $$C'_H=2\Gamma(2-2H)\sin(\pi H - \frac 12 \pi).$$
\end{enumerate}
\end{thm}
The following theorem is well-known Wick formula (see \cite{Nourdin2012} p.202 (A.2.13) and \cite{Janson1997} Theorem 1.28).
\begin{thm}\label{duochong}
Given a finite set $b$. Denote by $\mathcal{P}(b)$ the class of all partitions of $b$. Assume that $(Y_1,\cdots , Y_n)$ is a zero mean jointly normal random vector, then
\begin{equation*}
    E[Y_1\cdots Y_n]=\left\{
      \begin{array}{ll}
 \sum_{\pi=\{i_k,j_k \}\in \mathcal{P}(n)} \prod_k E[Y_{i_k}Y_{j_k}], & \quad n ~\text{is even }  ;\\
0, &\quad n~\text{is odd }.
 \end{array}
\right.
\end{equation*}Observe that, on the right-hand side, the sum is taken over all partitions $\pi$ of $n$ such that each block of $\pi$ contains exactly two elements.
\end{thm}
We denote by $\mathcal{M}_d(\mathbb{R})$ the collection of all real $d\times d$ matrices and $\mathcal{N}_d(0, C)$ the law of an $\mathbb{R}^d$-valued Gaussian vector with zero mean and covariance matrix $C$.

Denote $\mathfrak{H}^{\otimes p}$and $\mathfrak{H}^{\odot p}$ as the $p$th tensor product and the $p$th symmetric tensor product of the Hilbert space $\mathfrak{H}$. Let $\mathcal{H}_p$ be $p$th Wiener chaos with respect to $G$. It is defined as the closed linear subspace of $L^2(\Omega)$ generated by the random variables $\{ H_p (G(h)):h\in \mathfrak{H},\parallel h\parallel_{\mathfrak{H}}=1 \}$, where $H_p$ is the $p$th Hermite polynomial defined by
$$H_p(x)=\frac{(-1)^p}{p!}e^{\frac{x^2}2} \frac{d^p}{d x^p} e^{-\frac{x^2}2}, \quad p \ge 1.$$
and $H_0(x)=1$. We have the identity $I_p(h^{\otimes p})=H_p(G(h))$ for any $h\in \mathfrak{H}$. where $I_p(\cdot)$ is the generalized Wiener-It\^{o} stochastic integral. Then the map $I_p$ provides a linear isometry between $\mathfrak{H}^{\odot p}$ (equipped with norm $\frac1{\sqrt{p!}}\parallel \cdot \parallel_{\mathfrak{H}^{\otimes p}}$) and $\mathcal{H}_p$. Here $\mathcal{H}_0 =\mathbb{R}$ and $I_0 (x)=x$, by convention.

The following theorem is multivariate fourth-moment theorem, see \cite{NourdinI2012} Theorem 6.5 or \cite{Nourdin2012} Theorem 6.2.3.
\begin{thm}\label{duochong1}
Let $d \ge 2$ and $q_d, . . . , q_1 \ge 1$ be some fixed integers. Consider vectors
$$F_n=(F_{1,n},\cdots,F_{d,n})=(I_{q_1}(f_{1,n}),\cdots,I_{q_d}(f_{d,n})),\quad   n\ge 1,$$
where $f_{i,n} \in L^2(\mathbb{R}^{q_i} )$ is symmetric function. Let $C \in \mathcal{M}_d (\mathbb{R}) $ be a symmetric matrix. Assume
that $$ \lim_{n\to \infty} E[F_{i,n}F_{j,n}]=C_{i,j}, \quad 1\le i,j \le d.$$ as $n\to \infty$, the following two conditions are equivalent:
\begin{enumerate}
  \item[(1)] $F_n$ converges in law to $\mathcal{N}_d(0, C)$;
 \item[(2)] For every $1 \le i \le d$, we have $F_{i,n}$ converges in law to $\mathcal{N} (0, C_{i,i})$.
\end{enumerate}
\end{thm}

\section{Proofs of The Strong Consistency}\label{J}
\noindent{\it Proof of Theorem~\ref{consistent} \,} First, we consider the strong consistency of $\bar{X}$. From Theorem 1.2 in \cite{Chen2023}, then $(Y_t, t \in \mathbb{N}) $  is a stationary and ergodic Gaussian process with zero mean. The ergodicity theorem implies
$$\frac 1n \sum_{t=1}^{n} Y_t \longrightarrow E(Y_t)=0 \quad \text{a.s.}, \quad n \to \infty.$$ Since
$$\lim_{n\to \infty} \frac 1n \zeta \sum_{t=1}^n {\theta}^t =0$$
and \eqref{Xt}, we have
\begin{align}\label{barX}
\bar{X} \longrightarrow \mu \quad \text{a.s.}, \quad n \to \infty.
\end{align}

To show the strong consistency of $\hat{\theta}_n$, we will consider $\frac 1n \sum_{t=1}^n (X_t -\bar{X})^2$. It is clear that
\begin{align}\label{barX2}
\frac 1n \sum_{t=1}^n (X_t -\bar{X})^2 = \frac 1n \sum_{t=1}^n (X_t -\mu)^2 - (\bar{X}-\mu)^2 = \frac 1n \sum_{t=1}^n (Y_t + \zeta\theta^t)^2 - (\bar{X} -\mu)^2.
\end{align}By Theorem 1.2 in \cite{Chen2023}, we have
$$\frac 1n \sum_{t=1}^n (X_t -\mu)^2 \longrightarrow f(\theta) \quad \text{a.s.}, \quad n\to \infty.$$
Substituting last formula and \eqref{barX} into \eqref{barX2}, we obtain
$$\frac 1n \sum_{t=1}^n (X_t -\bar{X})^2 \longrightarrow f(\theta) \quad \text{a.s.}, \quad n\to \infty.$$
By the second moment estimator of $\theta$ (see \eqref{thetaestimator})
$$\hat{\theta}_n = f^{-1} \left(\frac 1n \sum_{t=1}^n (X_t -\bar{X})^2 \right),$$
we deduce that $f^{-1}(\cdot)$ is a continue function. The continuous mapping theorem implies
\begin{align}\label{jielun1}
\lim_{n\to \infty} \hat{\theta}_n =\theta \quad \text{a.s.}
\end{align}

Finally, \eqref{barX} and \eqref{jielun1} imply
$$(\bar{X},\hat{\theta}_n) \longrightarrow (\mu,\theta) \quad \text{a.s.}, \quad n\to \infty. $$Let
$$g(\mu,\theta)=(1-\theta)\mu,$$
Again by the continuous mapping theorem, we conclude that $\hat{\alpha}_n$ is strongly consistent.

\section{The Asymptotic Normality}\label{J2}
Define
\begin{align}\label{Vn}
V_n:= \frac 1{\sqrt{n}} \sum_{k=1}^n [Y_k^2 - E(Y_k^2)], \quad n \ge 1.
\end{align}

See Lemma 3.3-3.4 in \cite{Chen2023}.
\begin{prop}\label{prop1}
Under Hypothesis~\ref{hyp1}, $Y_t$ given in \eqref{Yt} exhibits linear long-range dependence. Namely,
\begin{align}\label{asymptotic1}
h_{Y}(\lambda) \backsim C_{\theta,H} |\lambda|^{1-2H}, \quad \text{as}~\lambda\to 0,
\end{align}
\end{prop}
where $C_{\theta,H} = \frac {\Gamma(2H-1) \sin(\pi-\pi H)}{(1-\theta)^2 \pi}$. If $\frac 12 < H < 1$, then the covariation of $Y_t$ satisfies
\begin{align}\label{asymptotic2}
R_Y(k) \backsim C_{\theta,H} |k|^{2H-2},\quad \text{as}~k \to \infty,
\end{align}where $C_{\theta,H} = \frac{B(2H-1,2-2H) \sin(2\pi H-\pi)} {(1-\theta)^2 \pi}$ and $B$ is Beta function.
\begin{prop}\label{prop2}
Let $\bar{X}$ and $\sigma_1^2$ be given in \eqref{hatalpha} and Theorem~\ref{asymptotic}(2) respectively. If $H > \frac 12$ , then as $n\to \infty $,
$$n^{1-H}(\bar{X}-\mu) \xrightarrow{law} \mathcal{N} \left(0, \frac {\sigma_1^2}{(1-\theta)^2} \right). $$

See \cite{NourdinI2012} Theorem 7.3(1) and \cite{Taqqu1975}.
\end{prop}
\begin{proof}
Define
\begin{align}\label{Zn}
Z_n:=n^{-H} \sum_{i=1}^n Y_i.
\end{align}Thanks to \eqref{Xt}, we have
$$n^{1-H} (\bar{X}-\mu)= Z_n + n^{-H}\zeta \sum_{t=1}^{n} \theta^t.$$
It is clear that
\begin{align}\label{Zn1}
n^{-H}\zeta \sum_{t=1}^{n} \theta^t \longrightarrow 0, \quad n\to \infty.
\end{align}Hence, we need only show that
\begin{align}\label{Zn2}
Z_n \stackrel{ {law}}{\longrightarrow} \mathcal{N} \left(0, \frac{\sigma_1^2}{(1-\theta)^2} \right) \quad n \to \infty.
\end{align}We claim that as $n \to \infty$
\begin{align}\label{Zn3}
E(Z_n^2) = \frac 1{n^{2H}} \sum_{k,j=1}^n R_Y (k-j) \longrightarrow \frac {\sigma_1^2}{(1-\theta)^2}.
\end{align}Indeed,decompose the left-hand side as
$$\frac 1{n^{2H}} \sum_{\substack{k,j=1,\cdots,n\\ |k-j|\le 2}} R_Y(k-j) + \frac 1{n^{2H}} \sum_{\substack{k,j=1,\cdots,n\\ |k-j|\ge 3}} R_Y(k-j).$$
Using $|R_Y (k)| \le f(\theta)$, we obtain the first term is bounded by $C \frac 1 {n^{2H-1}}$ and therefore tends to zero as $n \to \infty$ (recall that $\frac 12 < H < 1$).Concerning the second term, one has, \eqref{asymptotic2} implies
\begin{align*}
R_Y(k) \backsim C_{\theta,H} |k|^{2H-2},\quad \text{as}~k \to \infty,
\end{align*}where
$$C_{\theta,H}= \frac{B(2H-1,2-2H) \sin(2\pi H-\pi)} {(1-\theta)^2 \pi} = \frac{H(2H-1)\sigma_1^2}{(1-\theta)^2}.$$
Then
$$\frac 1{n^{2H}} \sum_{\substack{k,j=1,\cdots,n\\ |k-j|\ge 3}} R_Y(k-j) \backsim \frac{H(2H-1)\sigma_1^2}{(1-\theta)^2} l_n , $$where
$$l_n = \sum_{\substack{k,j=1,\cdots,n\\ |k-j|\ge 3}} \left| \frac {k-j}{n} \right|^{2H-2 } \frac 1{n^2}. $$
Therefore, there exists $l_n \to l_{\infty}$ for all $x \in (0, 1)$ , where
$$l_{\infty} = 2\int_{0}^{1}(1-x)x^{2H-2} \dif x.$$

If $g(x) = (1-x)x^{2H-2}$ with $H \in (\frac 12,1)$, it is obvious that it is decreasing monotonically at $(0, 1)$. Let us show that $l_n$ is dominated by an integral function and the convergence rate is obtained
\begin{align*}
    |l_n-l_{\infty}| &=2 \left| \sum_{j=3}^n \left(1-\frac jn \right) \left(\frac jn \right)^{2H-2} \frac 1n - \int_{0}^{1}g(x) \dif x \right| \\
    & =2\left| \sum_{j=3}^n \int_{\frac {j-1}n}^{\frac jn} g\left(\frac jn \right) \dif x -\sum_{j=1}^n \int_{\frac {j-1}n}^{\frac jn}g(x) \dif x \right| \\
    & \le 2\left|\int_{0}^{\frac 2n} f(x) \dif x + \sum_{j=3}^n \int_{\frac {j-1}n}^{\frac jn} \left[ g\left(\frac {j-1}n \right)-g\left(\frac jn \right) \right] \dif x   \right| \\
    & = C \frac 1{n^{2H-1}}.
\end{align*}Therefore, $l_n \to l_{\infty}$, as $n\to \infty$. And this concludes the proof of \eqref{Zn3}.

Characteristic function of $Z_n$
$$E(e^{i v Z_n}) = \exp \left(-\frac 12 v^2 E[Z_n^2] \right)  \xrightarrow{n \to \infty} \exp \left(-\frac 12 v^2 \frac {\sigma_1^2}{(1-\theta)^2} \right), $$
which implies the desired \eqref{Zn2}.
\end{proof}
\noindent{\it Proof of Theorem~\ref{asymptotic} \,} The limiting distribution of $\hat{\theta}_n$ is similar to AR(1) model with zero mean.
Consider
\begin{align}\label{Zn4}
\frac 1{\sqrt{n}} \sum_{n=1}^n \big((X_t-\bar{X} )^2 -f(\theta) \big)=\frac 1{\sqrt{n}} \sum_{n=1}^n \big((X_t-\mu )^2 -f(\theta) \big) -\sqrt{n} (\bar{X} - \mu)^2.
\end{align}By Theorem 1.3 in \cite{Chen2023}, we have
\begin{align}\label{Zn5}
\frac 1{\sqrt{n}} \sum_{n=1}^n \big((X_t-\mu )^2 -f(\theta) \big) \xrightarrow{law} \mathcal{N}(0, \sigma_H^2), \quad n \to \infty.
\end{align}For $H \in( \frac 12 , \frac 34 )$, Proposition~\ref{prop2} implies
\begin{align}\label{Zn8}
\sqrt{n}(\bar{X}-\mu)^2 = n^{2H-\frac 32} \big(n^{1-H}(\bar{X}-\mu)  \big)^2 \xrightarrow{n \to \infty} 0 .
\end{align}Substituting \eqref{Zn5} and \eqref{Zn8} into \eqref{Zn4}, Slutsky's theorem implies
\begin{align}\label{Zn9}
\frac 1{\sqrt{n}} \sum_{t=1}^n  \big((X_t-\bar{X})^2 - f(\theta)  \big) \xrightarrow{law} \mathcal{N}(0,\sigma_H^2), \quad n \to \infty.
\end{align} Delta method and \eqref{jielun1} imply
$$\left[f^{-1} \left( \frac 1n \sum_{t=1}^{n} (X_t - \bar{X})^2   \right)  \right]' = \frac{1}{f'(\hat{\theta}_n)} \xrightarrow{n \to \infty}\frac{1}{f'(\theta)}  \quad \text{a.s.} $$
Thus,
$$\sqrt{n}(\hat{\theta}_n - \theta) \xrightarrow{law} \mathcal{N} \left(0, \frac {\sigma_H^2}{[f'(\theta)]^2} \right).$$

Second, consider the asymptotic normality of $\hat{\alpha}_n$. The difference of $\hat{\alpha}_n$ and $\alpha$ are represented as the follow:
\begin{align*}
    \hat{\alpha}_n - \alpha &= (1-\hat{\theta}_n)\bar{X} - (1-\theta)\mu \\
    &= (1-\hat{\theta}_n)\bar{X}- (1-\theta)\bar{X} + (1-\theta)\bar{X} - (1-\theta)\mu \\
    &= (\theta -\hat{\theta}_n) \bar{X} +(1-\theta) (\bar{X}-\mu).
\end{align*}If $H >\frac 12$, then
\begin{align*}
    \lim_{n\to \infty} n^{1-H} (\hat{\theta}_n - \theta) \bar{X} = \lim_{n \to \infty} \frac{\sqrt{n}}{n^{H-\frac 12}} (\hat{\theta}_n -\theta) \bar{X} = \mu \lim_{n\to \infty} \frac{\sqrt{n}}{n^{H-\frac 12}} (\hat{\theta}_n -\theta) =0.
\end{align*}Therefore,
\begin{align}\label{Zn6}
\lim_{n \to \infty} n^{1-H} (\hat{\alpha}_n -\alpha) =(1-\theta) \lim_{n \to \infty} n^{1-H}(\bar{X}- \mu).
\end{align}Proposition~\ref{prop2} and Slutsky's theorem imply
$$n^{1-H} (\hat{\alpha}_n -\alpha) \xrightarrow{law} \mathcal{N}(0,\sigma_1^2), \quad n\to \infty.$$

Now, consider the jointly asymptotic normality of $\hat{\theta}_n$ and $\hat{\alpha}_n$. $Z_n$ and $V_n$ are given by \eqref{Zn} and \eqref{Vn} respectively. Theorem~\ref{duochong} implies
$$E(Z_n V_n)=n^{-\frac 12-H} \sum_{i=1}^n \left[ \sum_{j=1}^n E(Y_i Y_j^2) - nf(\theta) E(Y_i)  \right]=0.$$
Then, $Z_n$ and $V_n$ are independent. Combining Theorem~\ref{duochong1} with proof of Theorem 1.3 in \cite{Chen2023}, \eqref{Zn2}, we obtain
$$(V_n,Z_n) \xrightarrow{law} \mathcal{N}_2 (0, \Sigma_1), \quad n\to \infty,$$
where
$$\Sigma = \begin{pmatrix}
 \sigma_H^2 & 0 \\
 0 & \frac{\sigma_1^2}{(1 -\theta)^2}
\end{pmatrix}.$$Denote
$$G_{n,3} := \frac 1{\sqrt{n}} \sum_{t=1}^n \big( (X_t - \bar{X})^2 - f(\theta) \big)$$and
\begin{align}\label{Zn7}
G_{n,4} := n^{1-H} (\bar{X}-\mu).
\end{align}Slutsky's theorem together with (3.11)-(3.12) in \cite{Chen2023}, \eqref{Zn1}, \eqref{Zn8} implies
$$(G_{n,3},G_{n,4}) \xrightarrow{law} \mathcal{N}_2(0, \Sigma_1), \quad n\to \infty.$$
Next, since the derivability of $f(\theta)$ at $\theta$, we have
$$f'(\theta) (\hat{\theta}_n -\theta)+ (\hat{\theta}_n -\theta) = \frac 1n \sum_{t=1}^n  (X_t - \bar{X})^2 - f(\theta), $$
$$f'(\theta)\sqrt{n} (\hat{\theta}_n -\theta)+ \sqrt{n}(\hat{\theta}_n -\theta) \frac {o(\hat{\theta}_n -\theta)}{\hat{\theta}_n -\theta}=\frac 1{\sqrt{n}} \sum_{t=1}^n \big( (X_t - \bar{X})^2 - f(\theta) \big).$$
Theorem~\ref{asymptotic}(1) implies that $\sqrt{n} (\hat{\theta}_n -\theta)$ converges in law. Thanks to \eqref{jielun1}, we have
$$\frac {o(\hat{\theta}_n -\theta)}{\hat{\theta}_n -\theta} \longrightarrow  0 \quad \text{a.s.}$$Then
$$\sqrt{n}(\hat{\theta}_n -\theta) \frac {o(\hat{\theta}_n -\theta)}{\hat{\theta}_n -\theta} \xrightarrow{law} 0 .$$Therefore,
$$\lim_{n\to \infty}\sqrt{n}(\hat{\theta}_n -\theta) = \lim_{n\to \infty} \frac 1{f'(\theta)} \frac{1}{\sqrt{n}} \sum_{t=1}^n \big( (X_t-\bar{X})^2 -f(\theta)  \big).$$

Using Cram\'er-Wold Device (\cite{Durrett2010} Theorem 3.9.5) together with \eqref{Zn6}, we conclude
\begin{align*}
    \lim_{n \to \infty} E \left( e^{i(v_1 G_{n,1}+v_2 G_{n,2})} \right) &= \lim_{n \to \infty} E \left( e^{i(v_1\frac 1{f'(\theta)} G_{n,3}+v_2(1-\theta) G_{n,4})} \right) \\
    &= \lim_{n \to \infty} \exp\left[-\frac 12 \left(v_1^2\frac 1{[f'(\theta)]^2} E[G_{n,3}^2]+v_2^2(1-\theta)^2 E[G_{n,4}]^2\right) \right] \\
    &= \exp\left[-\frac 12 \left(v_1^2\frac {\sigma_H^2}{[f'(\theta)]^2} +v_2^2\sigma_1^2\right) \right].
\end{align*}Then
$$(G_{n,1},G_{n,2}) \xrightarrow{law} \mathcal{N}_2 (0,\Sigma), \quad n\to \infty.$$

\section{Simulation}\label{J3}
In this section, we will test the performance of estimators $\hat{\theta}_n$ and $\hat{\alpha}_n$ under limited samples, and test their joint distribution performance. By the delta method, we know that the asymptotic normal of the moment estimator $\hat{\theta}_n$ is equivalent to that of the random variable
$$\frac 1{\sqrt{n}} \sum_{n=1}^n \big((X_t-\bar{X} )^2 -f(\theta) \big),$$
please refer to \eqref{Zn9}. Since the inverse function $\theta = f^{-1}(\cdot)$ does not have an analytic expression, it is convenience to plot the figure of the above random variable to show the asymptotic normality. The fGn noise and the ARFIMA(0 ,d ,0) noise data used in this section are all generated by the R package "longmemo" in \cite{Beran1994}.

In this experiment, the data generated with $\theta = 0.6$, $n = 3000$, and each group of data will be copied $M = 10000$ times by Monte Carlo method.In order to better test the estimator, we simulate fGn noise with $H = 0.58$ (see subscript (a) of all the following figures) and the ARFIMA(0 ,d ,0) noise with $d = 0.08$ (see subscript (b) of all the following figures) by Monte Carlo method. We have the asymptotic normality for that random variable, as shown in Figure.\ref{tu2}. And we have the asymptotic normality for that the estimator $\hat{\alpha}_n$, as shown in Figure.\ref{tu3}. Thus the joint asymptotic normality of the two moment estimators $\hat{\theta}_n$ and $\hat{\alpha}_n$, as shown in Figure.\ref{tu4}.
\begin{figure}[htbp]
    \centering
    \subfigure[]{
    \includegraphics[width=5.5cm]{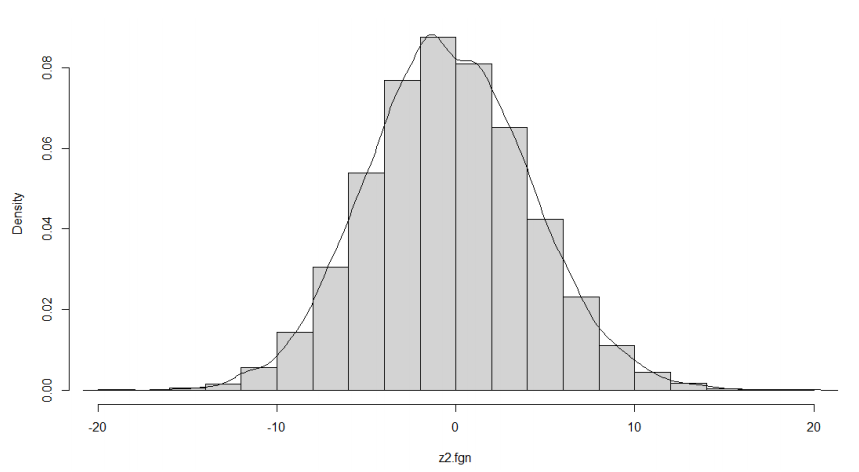}
    }
    \subfigure[]{
    \includegraphics[width=5.5cm]{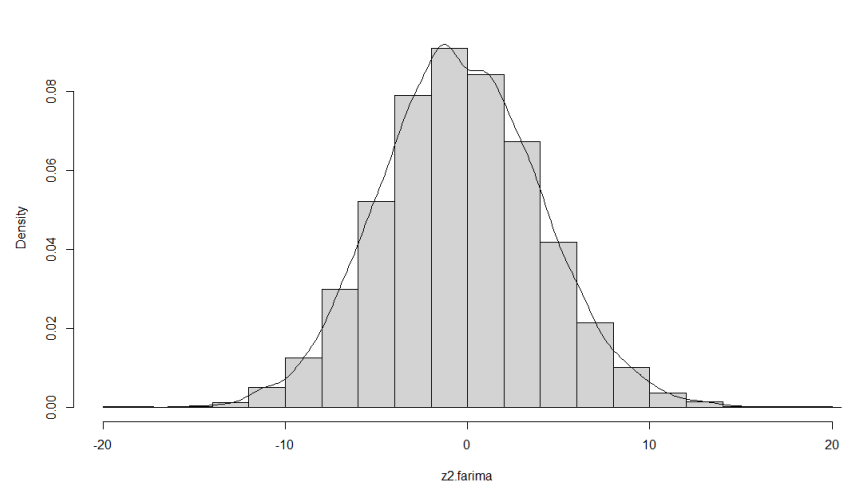}
    }
    \caption{The asymptotic normality of $\frac 1n \sum_{t=1}^n (X_t - \bar{X})^2 $.}
    \label{tu2}
\end{figure}
\begin{figure}[htbp]
    \centering
    \subfigure[]{
    \includegraphics[width=5.5cm]{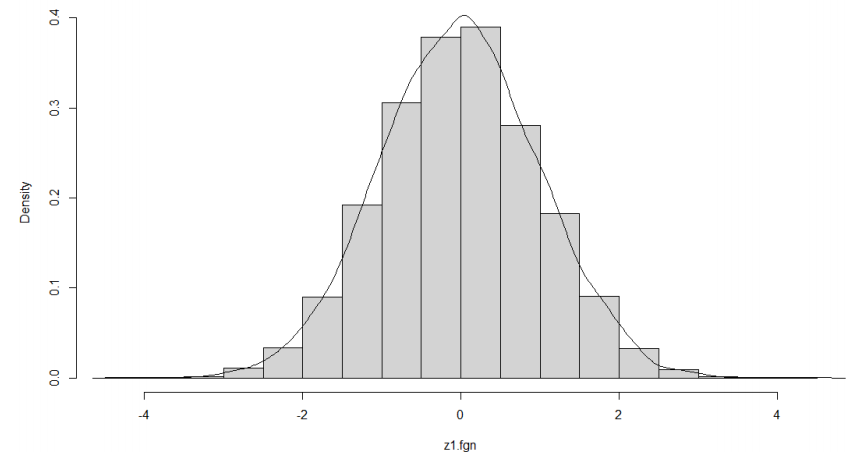}
    }
    \subfigure[]{
    \includegraphics[width=5.5cm]{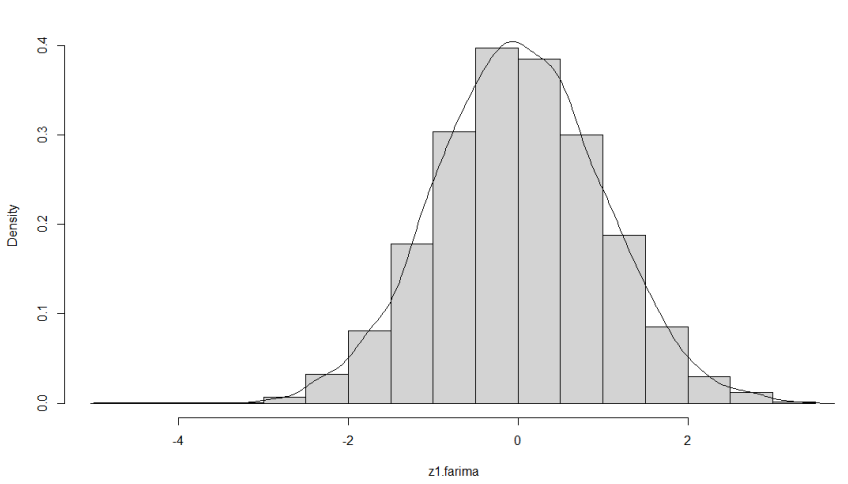}
    }
    \caption{The asymptotic normality of $\hat{\alpha}_n $.}
    \label{tu3}
\end{figure}
\begin{figure}[htbp]
    \centering
    \subfigure[]{
    \includegraphics[width=5.5cm]{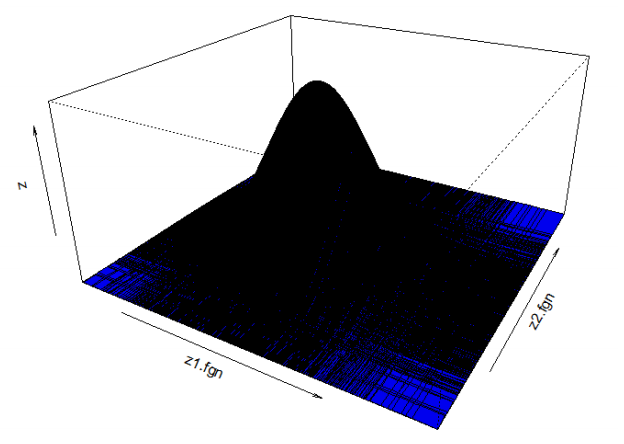}
    }
    \subfigure[]{
    \includegraphics[width=5.5cm]{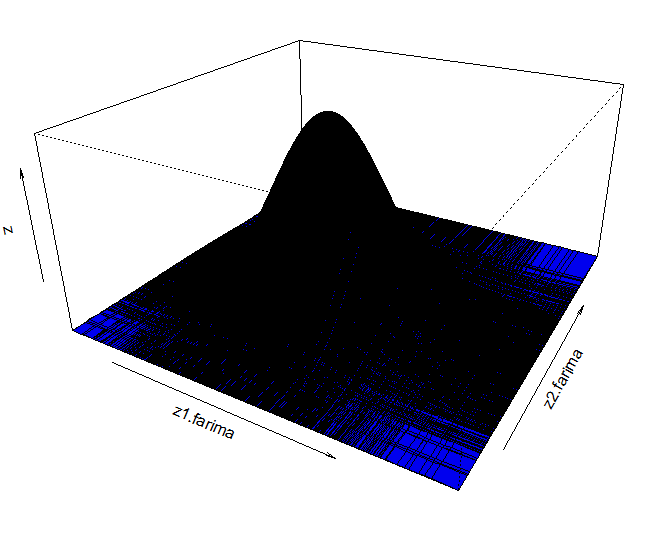}
    }
    \caption{The joint asymptotic distributions of $\frac 1n \sum_{t=1}^n (X_t - \bar{X})^2 $ and $\hat{\alpha}_n $.}
    \label{tu4}
\end{figure}

%-------------------------------------------------------------------------------------------------------
%-------------------------------------------------------------------------------------------------------
\end{document}